\documentclass[a4paper,12pt,reqno]{amsart}
\usepackage{latexsym, amsfonts, amsmath, amsthm, amssymb, enumerate, mathtools, tikz-cd, verbatim}
\usepackage[mathscr]{euscript}
\usepackage[top=3cm, bottom=3cm, left=2.54cm, right=2.54cm]{geometry}
\allowdisplaybreaks

\renewcommand\Re{\mathop{{\rm Re}}}

\newcommand{\Db}{{\mathbb D}}
\newcommand{\Dbar}{{\overline{\Db}}}
\newcommand{\Cb}{{\mathbb C}}
\newcommand{\Nb}{{\mathbb N}}
\newcommand{\Rb}{{\mathbb R}}

\newcommand{\calD}{\mathcal{D}}
\newcommand{\calL}{\mathcal{L}}

\newcommand{\calS}{\mathcal{S}}

\def\colvec[#1,#2]{\begin{bmatrix} #1 \\ #2 \end{bmatrix}}
\def\rowvec[#1,#2]{\begin{bmatrix} #1 & #2 \end{bmatrix}}
\def\ip<#1,#2>{\left\langle #1,#2 \right\rangle}
\newcommand{\norm}[1]{\left\Vert #1 \right\Vert}

\newcommand{\cl}{\mathop{\mathrm{cl}}}

\newcommand{\degree}{\mathop{\mathrm{degree}}}
\newcommand{\Der}{\mathop{\mathrm{Der}}}
\newcommand{\Hom}{\mathop{\mathrm{Hom}}}
\newcommand{\HT}{\mathop{\mathcal{H}}}
\newcommand{\pv}{\mathop{\mathrm{p.v.}}}

\newcommand{\trace}{\mathop{\mathrm{trace}}}
\newcommand{\st}{\,:\,}

\newtheorem{thm}{Theorem}[section]
\newtheorem*{thm*}{Theorem}
\newtheorem{cor}[thm]{Corollary}
\newtheorem*{cor*}{Corollary}
\newtheorem{lem}[thm]{Lemma}
\newtheorem{prop}[thm]{Proposition}

\newtheorem*{con*}{Conjecture}
\newtheorem*{prob*}{Problem}

\theoremstyle{definition}
\newtheorem{defn}[thm]{Definition}
\theoremstyle{remark}
\newtheorem{rem}[thm]{Remark}

\newtheorem{ex}[thm]{Example}

\numberwithin{equation}{section}

\newcommand{\idcom}[1]{\footnote{\textcolor[rgb]{0.98,0.00,0.00}{ID: #1}}}

\begin{document}


\title{Linear systems, Hankel products and the sinh-Gordon equation}

\author{Gordon Blower}
\email{g.blower@lancaster.ac.uk}
\address{Department of Mathematics and Statistics, Lancaster
University, Lancaster LA1 4YF, United Kingdom}
\author[Ian Doust]{Ian Doust}
\email{i.doust@unsw.edu.au}
\address{School of Mathematics and Statistics, University of New South Wales, Sydney, NSW 2052, Australia}

\subjclass[2010]{}

\date{13 October 2022}

\begin{abstract} 
Let $(-A,B,C)$ be a linear system in continuous time $t>0$ with input and output space ${\mathbb C}^2$ and state space $H$. The scattering functions $\phi_{(x)}(t)=Ce^{-(t+2x)A}B$ determines a Hankel integral operator $\Gamma_{\phi_{(x)}}$; if $\Gamma_{\phi_{(x)}}$ is trace class, then the Fredholm determinant $\tau (x)=\det (I+\Gamma_{\phi_{(x)}})$ determines the tau function of $(-A,B,C)$. The paper establishes properties of algebras including $R_x=\int_x^\infty e^{-tA}BCe^{-tA}dt$ on $H$. Thus the paper obtains solutions of the sinh-Gordon PDE. The tau function for sinh-Gordon satisfies a particular Painl\'eve $\mathrm{III}'$ nonlinear ODE and describes a random matrix model, with asymptotic distribution found by the Coulomb fluid method to be the solution of an electrostatic variational problem on an interval.   
\end{abstract}

\maketitle

\par

\section{Introduction}\label{S:Intro}
This paper is concerned with the Fredholm determinants of operators that are introduced via linear systems and with their applications to the sinh-Gordon equations. The study has application in Tracy and Widom's approach \cite{TW} to random matrix theory. 

We begin by fixing some notation concerning linear systems. 
Let $H$ be a separable complex Hilbert space with orthonormal basis $(e_j)_{j=0}^\infty$, and let ${\mathcal L}(H)$ denote the algebra of bounded linear operators on $H$. We shall denote the adjoint of $B \in \calL(H)$ by $B^\dagger$. Let $H_0$ be a separable complex Hilbert space which serves as the input and output space; let $B:H_0\rightarrow H$ and $C:H\rightarrow H_0$ be bounded linear operators. 
On the state space $H$, let $(T_t)_{t>0}$ be a strongly continuous and bounded semigroup with infinitesimal generator $-A$, which is densely defined on domain ${\mathcal D}(A)$, where ${\mathcal D}(A)$ is itself a Hilbert space for the graph norm 
$\Vert f\Vert^2_{{\mathcal{D}}(A)} =\Vert f\Vert^2+\Vert Af\Vert^2$.  

The continuous time linear system  $(-A,B,C)$ is
\begin{align}\label{linearsystem}
  {\frac{dX}{dt}} &=-AX+BU\nonumber\\
            Y     &= CX,\nonumber \\
          X(0)    &= 0.
    \end{align}
The scattering function of $(-A,B,C)$ is $\phi (t) = Ce^{-tA}B$, which is a bounded and weakly continuous function $\phi :(0, \infty ) \rightarrow {\calL}(H_0)$.

Suppose that $\phi\in L^2((0, \infty );{\mathcal L}(H_0))$. Then the Hankel integral operator  with scattering function 
$\phi$ is the operator
  \begin{equation}\label{eq2.29}
    \Gamma_\phi f(x) =
       \int_0^\infty \phi(x+y) f(y) \,dy
              \qquad (f \in L^2((0,\infty); H_0)).
  \end{equation}

Such Hankel operators do not themselves form an algebra, although they have an algebraic structure which has been exploited in \cite{Po} and section 3.5 of \cite{McK}. An important fact is that every bounded self-adjoint Hankel integral operator on $L^2(0, \infty )$ can be realised as the Hankel operator associated with a linear system $(-A,B,C)$ in continuous time with state space $H$, and such that $\phi (t)=Ce^{-tA}B$ where $(e^{-tA})$ is a strongly continuous semigroup on $H$ (see \cite{Pe} and \cite{MPT}).
 
Recall that $K \in \calL(H)$ is Hilbert--Schmidt if
$\norm{K}_{{\calL}^2}^2 = \sum_{j=0}^\infty \norm{ Ke_j }^2$
 is finite. We shall denote the set of all Hilbert--Schmidt operators on $H$ by ${\calL}^2 = {\calL}^2(H)$. This space contains the ideal of trace class operators ${\calL}^1=\{ K_1 K_2 \st K_1, K_2\in {\calL}^2\}$. 
There are now several criteria for ensuring the boundedness of a Hankel operator; see, for example \cite{Pe} and \cite{Power}. We shall repeatedly use the basic result that  
if $t\Vert \phi (t)\Vert^2_{{\mathcal L}^2(H_0,H_0)}$ is integrable, then $\Gamma_\phi$ determines a Hilbert--Schmidt operator $L^2((0,\infty ); H_0)\rightarrow L^2((0, \infty ); H_0)$. 

 \begin{defn} Let $(-A,B,C)$ be a linear system as in (\ref{linearsystem}) with scattering function $\phi$, and suppose that $\Gamma_\phi$ is trace class.  Let $\phi_{(x)}(t)=\phi (t+2x)$. 
 Then the tau function of $(-A,B,C)$ is defined to be  $\tau(x)=\det (I+\Gamma_{\phi_{(x)}})$, for $x > 0$.
 \end{defn}
 
 This tau function is analogous to the tau function introduced by Jimbo, Miwa and Ueno \cite{JMU} to describe the isomonodromy of rational differential equations, and generalizes the classical theta function. Mumford \cite{M} constructed solutions of the cubic nonlinear Schr\"odinger, $KdV$, $mKdV$ and sine-Gordon equations in terms of classical theta functions on abelian varieties. Ercolani and McKean \cite{EMcK} extended the analysis of $KdV$ to infinite-dimensional abelian varieties in which case the scattering function $\phi$ is of rapid decay. 

In \cite{BN}, we reinterpreted the results of \cite{Po} and \cite{McK} in terms of an algebra of operators on $H$ with a special associative product, and showed how nonlinear differential equations such as $KdV$ emerge from algebraic identities in this product. In the present paper, we continue this analysis by addressing the sinh-Gordon equation in section~\ref{S:Sinh-Gordon}.
The simplest Darboux transformation of a linear systems is $(-A,B,C)\mapsto (-A,B,-C)$ which takes $R\mapsto -R$. In section \ref{S:Sinh-Gordon}, we introduce a $2\times 2$ block matrix system which enables use to introduce a tau function for the sinh-Gordon equation in the form $\det (I+\Gamma_{\phi_{(x)}})(I-\Gamma_{\phi_{(x)}})^{-1}$.\par
\indent

In \cite{DIZ}, the authors consider the algebra of integrable operators of the form $\lambda I+K$ where $K$ is an integral operator on $L^2(L)$ for a curve $L$ in ${\mathbb C}$ with kernel
  \[
  k_K(z,w)={\frac{\sum_{j=1}^N f_j(z)g_j(w)}{z-w}}
  \]
where $\sum_{j=1}^N f_j(z)g_j(z)=0$. They achieve several results using Riemann--Hilbert theory and make applications to some integrable operators in random matrix theory. Integrable operators of this form include the Christoffel--Darboux kernel from orthogonal polynomials when $L$ is a real subinterval. In \cite{TW}, the authors systematically consider these kernels and their scaling limits for classical orthogonal polynomials and obtain integrable operators that are fundamental to random matrix theory. In this application, the Fredholm determinant $\det (I+\lambda K)$ is of primary interest. In particular, they show that the Airy and Bessel kernels can be expressed as products of Hankel integral operators on $L^2(0, \infty )$. 

In section \ref{S:Hankel products} we give sufficient conditions $T\in {\mathcal A}$ to be expressed as a product of Hankel integral operators.

\section{Differential rings and Darboux addition}\label{S:Diff-rings}

Let $u\in C^\infty ({\mathbb R}; {\mathbb C})$ be the potential in Schr\"odinger's equation. Gelfand and Dikii \cite{GD} considered the algebra ${\mathcal A}={\mathbb C}[u, u', \dots ]$ of differential polynomials generated by the potential and its derivatives. In this section we consider a generalization of this relating to linear systems. In \cite{MPT} there is an existence theory covering the self-adjoint case. 

We suppose that $\phi$ is a scalar scattering function that can be realized from a linear system $(-A,B,C)$ from (\ref{linearsystem}) as $\phi (t)=Ce^{-tA}B$, where $A:H\rightarrow H$, $B:{\Cb} \rightarrow H$ and 
$C:H\rightarrow {\Cb}$ are bounded. Suppose further that $\Xi_x, \Theta_x :L^2((0, \infty ); H_0^2)\rightarrow L^2((0, \infty );{\mathbb C})$ are Hilbert--Schmidt operators, where
\begin{align}\label{controllability}
    \Xi_x f &= \int_x^\infty e^{-tA}Bf(t)\, dt\\
  \Theta_x f&= \int_x^\infty e^{-tA^\dagger}C^\dagger f(t)\, dt.
\end{align}
Then 
\begin{equation}\label{R}
   R_x = \Xi_x\Theta^\dagger_x
        = \int_x^\infty e^{-tA}BCe^{-tA} \,dt
\end{equation}
is trace class, and satisfies 
the Lyapunov equation
\begin{equation}\label{Lyapunov}
   {\frac{d}{dx}}R_x=-AR_x-R_xA.
\end{equation}

\begin{defn}  The tau function of $(-A,B,C)$ is $\tau (x)=\det (I+R_x)$.
\end{defn}

 With $F_x=(I+R_x)^{-1}$,  (\ref{Lyapunov}) gives
  \[ (d/dx)F_x=AF_x+F_xA-2F_xAF_x=e^{-xA}BCe^{-xA}. \]
 We introduce the vector space of linear operators
\begin{equation}\label{diffring}
   {\mathcal A}_{\Sigma} =
     {\hbox{span}}_{\Cb} \{ A^{n_1}, A^{n_2}F_xA^{n_2}\dots F_xA^{n_r}; n_j\in {\Nb}\},
\end{equation}
and introduce the associative product
\begin{equation}\label{ast}
   P\ast Q = P(A F_x + F_x A - 2 F_x A F_x)Q\qquad (P,Q \in {\mathcal A}_\Sigma )
\end{equation}
and the derivation
\begin{equation}\label{partial}
   \partial P =A(I-2F_x)P+{\frac{dP}{dx}}+P(I-2F_x)A \qquad (P\in {\mathcal A}_\Sigma ).
\end{equation}
and the bracket operation
 \begin{equation}\label{bracket}
  \bigl\lfloor P\bigr\rfloor 
      = Ce^{-xA}F_xPF_xe^{-xA}B,\qquad (P\in {\mathcal A}_\Sigma )
 \end{equation}
as in \cite{BN} Definition 4.3.

\begin{lem}\label{diffringlemma} 
\begin{enumerate}
    \item[(i)] $({\mathcal A}_\Sigma , \ast, \partial )$ is a differential ring, and the bracket operation  $\lfloor \, \cdot \rfloor : ({\mathcal A}_\Sigma, \ast , \partial )\rightarrow (C^\infty ((0, \infty ); \Cb), \cdot , d/dx)$ gives a homomorphism of differential rings.
    \item[(ii)] Let the potential of the linear system $(-A,B,C)$ be 
     \[ u=-2 \, \frac{d^2}{dx^2} \log\det (I+R_x). \] 
    Then $u=-4\lfloor A\rfloor$.
\end{enumerate}
\end{lem}

\begin{proof} See \cite{BN} Lemma 4.2, and a related approach is developed in \cite{DMS}. 
\end{proof}

\begin{defn} \cite{GH} 
\begin{enumerate}
\item[(i)] For a potential $u\in C^\infty ({\mathbb R}; {\mathbb C})$ the stationary $KdV$ hierarchy for a sequence $(f_\ell )$ is the recurrence relation $f_0=1,$
\begin{equation}\label{KdVrecurrence}
  {\frac{\partial f_\ell}{\partial x}}
   = -{\frac{1}{4}}{\frac{\partial^3 f_{\ell-1}}{\partial x^3}}
       +u{\frac{\partial f_{\ell-1}}{\partial x}}
       +{\frac{1}{2}} {\frac{\partial u}{\partial x}}f_{\ell -1}
       \qquad (\ell = 1,2,3,\dots).
\end{equation}
\item[(ii)] The stationary hierarchy for the sine-Gordon and modified $KdV$ equation is similar, with $u$ replaced by \begin{equation}\label{mKdVrecurrence}
    w_+=-{\frac{1}{4}}\Bigl( \Bigl({\frac{\partial u}{\partial x}}\Bigr)^2+2i{\frac{\partial^2u}{\partial x^2}}\Bigr).
\end{equation} 
\end{enumerate}
\end{defn}

The solutions sequence for these hierarchies are differential polynomials in $u$ by \cite{GH} Remark 2.2, and by Lemma \ref{diffringlemma} there exists $X_j\in {\mathcal A}_\sigma$ such that $f_\ell =\lfloor X_\ell\rfloor $. For the stationary $KdV$ hierarchy, we showed in \cite{BN} that when $f_0=(1/2)u$, the terms $f_\ell=(-1)^\ell 2\lfloor A^{2\ell-1}\rfloor$ give a solution.  In the physical models discussed in \cite{HMNP}, $KdV$ has a partition function which is the square of a tau function; whereas $mKdV$ has a partition function that is the product of tau functions; for sinh-Gordon, one considers quotients of tau functions.

Suppose that $u\in C^\infty (\Rb; \Rb)$ is bounded and let $\lambda_0$ be the bottom of the spectrum of $L=-d^2/dx^2+u$. For $\lambda < \min \{\lambda_0, 0\}$ and signs $\{ \pm \}$, we consider solutions:
\begin{enumerate}
 \item[($+$)] $h_{+}(x; \lambda )$ of $Lh_+=\lambda h_+$ such that $h_+\geq 0$, with
 \begin{equation}\int_0^\infty h_+(x)^2\,dx < \infty, \quad {\hbox{and}}\quad  \int_{-\infty}^0h_+(x)^2\, dx = \infty ;\end{equation}
 \item[($-$)] $h_{-}(x; \lambda )$ of $Lh_-=\lambda h_-$ such that $h_-\geq 0$, with
\begin{equation}\int_{-\infty}^0 h_-(x)^2\,dx < \infty\quad{\hbox{and}}\quad \int_0^{\infty}h_-(x)^2\,dx = \infty.\end{equation}
\end{enumerate}

Let $G(x,y;\lambda )$ be the Greens integral kernel that implements $(\lambda I-L)^{-1}$ on $L^2({\mathbb R}; {\mathbb C}),$ so the diagonal of $G$ satisfies 
\begin{equation}\label{Green}
   G(x,x;\lambda ) = {\frac{h_+(x, \lambda )h_-(x\, \lambda )} {{\hbox{Wr}}\,(h_+ (x, \lambda ), h_-(x, \lambda ))}}.
\end{equation}

In the case where $u$ is a real continuous and $2\pi$-periodic potential, the differential equation $-h''+uh=\lambda h$ is known as Hill's equation, and is associated with a hyperelliptic spectral curve $\Sigma$ with points $\{ (\lambda, \pm ): \lambda \in {\mathbb C}\}$ giving a two-sheeted cover of ${\mathbb C}$ as in \cite{MvM}. The notion of Darboux addition refers to the operation on potentials that corresponds to the addition rule for pole divisors of Baker-Akhiezer functions on $\Sigma$, as discussed in \cite{McK3}. In our case, we have a scattering potential $u$, and we consider the corresponding addition rule on the potentials, without seeking to interpret directly the notion of the spectral curve. Nevertheless, there is a simple operation on the linear system $(-A,B,C)$ that gives rise to Darboux addition on the associated potential from Lemma \ref{diffringlemma} (ii).   

\begin{defn} (Darboux addition) 
For $\sigma \in \{+,-\}$, we define $e(x,\lambda,\sigma) = h_\sigma(x,\lambda)$. 
Associated with $(\lambda, \sigma )$ there is a change in the potential
  \[ T^{(\lambda , \sigma )}:u\mapsto u-2{\frac{d^2}{dx^2}} \log e(x,\lambda, \sigma ).\]
corresponding to the change of linear systems
\[ T^{(\zeta,\sigma)}(-A,B,C)
     = (-A,(\zeta I + \sigma A)(\zeta I - \sigma A)^{-1} B,C).\]
\end{defn}

Let ${\hbox{Wr}}\,(f,g)=fg'-f'g$ be the Wronskian determinant of $f$ and $g$. Then for $(\lambda_1, \sigma_1)$ and $(\lambda_2, \sigma_2)$, the effect of successive transformations is
\begin{equation}\label{comp}T^{(\lambda_2, \sigma_2) }T^{(\lambda_1, \sigma_1)}: u\mapsto u-2{\frac{d^2}{dx^2}} \log {\hbox{Wr}}\,(e(x,\lambda_2, \sigma_2), e(x,\lambda_1, \sigma_1)).\end{equation}
See \cite{BN} Corollary 3.6 and \cite{McK2}. In particular, with $p_1=(\lambda , -)$ followed by $p_2=(\lambda +\delta \lambda , +)$, we have an infinitesimal addition $u\mapsto u-\delta \lambda {\bf X}(u)+o(\delta\lambda )$. We proceed to compute the generating function for infinitesimal addition. 

\begin{prop} Suppose that the linear system $(-A,B,C)$ has potential $u=-4\lfloor A\rfloor$. Then the infinitesimal addition satisfies
\begin{equation}\label{diaggreen} {\bf X}(u)=-2{\frac{d}{dx}}G(x,x; \lambda )\end{equation}
or equivalently
\begin{equation}
  {\bf X}(u) =
{\frac{-2}{\sqrt{-\lambda }}}\bigl\lfloor A(I-2F)A(\lambda I+A^2)^{-1} +A(\lambda I+A^2)^{-1}(I-2F)A\bigr\rfloor
       \qquad (\lambda <\lambda_0).\end{equation}
\end{prop}

\begin{proof} We begin by recalling McKean's calculation. By the composition formula (\ref{comp}), the infinitesimal addition satisfies
\begin{align}\label{comp2}u-\delta \lambda {\bf X}(u)+o(1)&=u-2{\frac{d^2}{dx^2}}\log {\hbox{Wr}}\,(h_+(x, \lambda +\delta\lambda ),h_-(x,\lambda ))\end{align}
where from the differential equation
\begin{equation}\label{Wronsk}{\frac{d}{dx}} {\hbox{Wr}}\,(h_+(x, \lambda +\delta\lambda ),h_-(x,\lambda ))=\delta \lambda h_+(x, \lambda +\delta\lambda )h_-(x, \lambda ).\end{equation}

Combining (\ref{comp2}), (\ref{Wronsk}) and (\ref{Green}), and taking $\delta\lambda\rightarrow 0$, one can express the infinitesimal transformation in terms of the diagonal of the Greens function as in (\ref{diaggreen}).
In \cite{BN} Theorem 5.4, we computed this diagonal Greens function, and we deduce
\[
  {\bf X}(u)  = {\frac{-2}{\sqrt{-\lambda }}}{\frac{d}{dx}}\Bigl( {\frac{1}{2}}-{\frac{\lfloor A\rfloor}{\lambda}}+ {\frac{\lfloor A^3\rfloor}{\lambda^2}}-
{\frac{\lfloor A^5\rfloor}{\lambda^3}}+\dots \Bigr)\qquad (\lambda\rightarrow-\infty )
\]
where the summands of this series are
  \[ {\frac{d}{dx}}\lfloor A^{2m+1}\rfloor =\lfloor A(I-2F)A^{2m+1}+A^{2m+1}(I-2F)A\rfloor \]
as in (5.4) of \cite{BN}. Hence by summing the resulting geometric series, we have
\[ {\bf X}(u) = {\frac{-2}{\sqrt{-\lambda }}}\bigl\lfloor A(I-2F)A(\lambda I+A^2)^{-1} +A(\lambda I+A^2)^{-1}(I-2F)A\bigr\rfloor.\]
\end{proof}

%
%

\section{Products of Hankel operators}\label{S:Hankel products}

In what follows we shall require some of the concepts of abstract differential calculus as developed in \cite{CQ}. 
Let $S$ be an associative and unital complex algebra. There is a natural multiplication map $\mu: S \otimes S \to S$, $\mu\bigl(\sum a_j \otimes b_j\bigr) = \sum a_j b_j$. We shall denote the nullspace of this map as $\Omega^1 S$ and so we have an exact sequence  
 \begin{equation}\label{Omega1} 0\longrightarrow \Omega^1 S
       \longrightarrow S\otimes S
       \longrightarrow S
       \longrightarrow 0;
 \end{equation}
In this section, we shall mainly be interested in the case where $S=\Cb(x)$, the space of rational functions. In these cases, the tensor product $S \otimes S$ may be regarded as consisting of functions of the form $\sum f_j(x) g_j(y)$ acting on a product space. There are several natural derivation maps that act on these spaces.

\begin{defn}\label{defn1.3} 
For $S = \Cb(x)$ define
\begin{enumerate}
  \item $d: S\rightarrow \Omega^1S$ $: f\mapsto df$ or $f(x)\mapsto f(x)-f(y)$, the noncommutative differential;
  \item $\delta :S\otimes S\rightarrow S\otimes S:$ $f(x,y)\mapsto (x-y)f(x,y)$, the inner derivation on $S\otimes S$ as an $S$-bimodule given by $x\in S$;
  \item $D: \Omega^1 S\to S\otimes S:$ $f dg\mapsto f(x){\frac{g(x)-g(y)}{x-y}}$, dividing by $x-y$; in particular, $D\circ d:s\rightarrow S\otimes S:$ $g(x)\mapsto {\frac{g(x)-g(y)}{x-y}}$ is the divided difference;
  \item $\partial :S\rightarrow S:$ $f(x)\mapsto f'(x)$, the usual derivative.
  \end{enumerate}
\end{defn}
See (\ref{diffvee}) for a significant example.\par
\indent Products of Hankel operators arise from rational differential equations, involving poles. As is common in control theory, we let ${\calS}$ be the unital complex algebra of stable rational functions, namely the space of $f(s)/g(s)$ where $f(s), g(s)\in {\Cb}[s]$ satisfy $\degree f(s)\leq \degree g(s)$, and all the zeroes of $g(s)$ satisfy $\Re s < 0$. 
We allow poles at $\infty$ by adjoining $s$ to form  ${\calS}[s]$. 
Now let $p_1, \dots, p_r \in \Cb$ be distinct points satisfying $\Re p_j > 0$, and for $P= \{ p_1, \dots ,p_r,\infty\}$, introduce the algebra 
  \begin{equation}\label{eq2.21}
  \calS_P= \calS  \left[s, \frac{s+1}{s-p_1}, \dots , \frac{s+1}{s-p_r} \right]
  \end{equation}
by formally adjoining the inverses of $(s-p_j)/(s+1)\in \calS$. The possible poles of elements of $\calS_p$ are in $P$ or the open left half plane.
For even $m=2n$, let $\Omega_\infty,\Omega_1,\dots,\Omega_r \in M_{m\times m}(\calS)$ and define $\Omega(s) \in M_{n\times n}(\calS_p)$ by
  \begin{equation}\label{eq2.22}
    \Omega (s)=\Omega_\infty(s)  s+\sum_{j=1}^r \frac{\Omega_j(s)}{s-p_j}.
  \end{equation}
The residues are $\Omega_\infty (\infty ), \Omega_1(p_1), \dots, \Omega_r(p_r)$. We assume further that $\Omega (s)^\top =\Omega (s)$, where $\Omega^\top$ denotes the transpose. Let
\begin{equation}\label{J}
    J=\begin{bmatrix}
      0   &   -I_n   \\
      I_n &     0 \end{bmatrix}.
\end{equation}
We also let $\langle \cdot, \cdot\rangle :{\mathbb C}^m\times {\mathbb C}^{m}\rightarrow {\mathbb C}$ $\langle (z_j)_{j=1}^m , (w_j)_{j=1}^m\rangle=\sum_{j=1}^mz_jw_j$ be the standard bilinear pairing, and $\Vert \cdot \Vert$ the operator norm on $M_{m\times m}({\Cb})$. 
Given a solution of the differential equation 
\begin{equation}
   J{\frac{d}{dx}}\Psi =\Omega \Psi,
\end{equation} 
for $\Psi \in L^2((0, \infty ); \Cb)$ we can introduce a kernel $k(x,y)=\langle J\Psi (x),\Psi (y)\rangle/ (x-y)$ which operates on $L^2((0, \infty ); \Cb)$. 

\begin{prop}\label{Prop2.8}
Let $\Omega$ be as in (\ref{eq2.22}) and suppose that
$\int_0^\infty t \norm{\Psi(t)}^2 \,dt$ is finite. Then there exist Hankel integral operators
  $\Gamma_{\psi_j},  \Gamma_{\phi_j}
       \in {\calL}^2(L^2((0, \infty ); M_{m\times m}({\Cb})))$
such that
  $K = \sum_{j=1}^N \Gamma_{\psi_j}\Gamma_{\phi_j}$ determines a trace class operator $K$ on $L^2((0, \infty ); {\Cb}^m)$ with kernel
  \begin{equation}\label{eq2.31}
     k(z,w) = \int_0^\infty \sum_{j=1}^N\psi_j(z+u) \phi_j(w+u) \,du
           \qquad (z,w>0).
  \end{equation}
\end{prop}

\begin{proof} Consider $\partial_\Delta: R\otimes R\rightarrow R\otimes R:$  $\partial_\Delta f(x,y) ={\frac{\partial}{\partial x}}f(x,y)
 + {\frac{\partial}{\partial y}}f(x,y)$, so $\partial_\Delta = \partial\otimes Id+Id\otimes\partial$.
We have
  \begin{align}
    \delta_\Delta K(z,w)
        &= \delta_\Delta D\bigl( \Psi(z)^\top J \Psi(w) - \Psi(w)^\top  J \Psi(z)\bigr)
             \notag \\
        &= \Psi(z)^\top \frac{-\Omega(z)+\Omega(w)}{z-w} \Psi(w)
             - \Psi(w)^\top \frac{\Omega(z)-\Omega(w)}{z-w} \Psi(z),
  \end{align}
which is a finite-rank $M_{m\times m}({\Cb})$-valued kernel. Here
the entries of the difference quotient $(\Omega(z)-\Omega (w))/(z-w)$ may be expressed as a sum of products of rational functions in $z$ and rational functions in $w$, which are bounded for $z,w>0$ by the partial fraction decomposition of $\Omega (s)$ in (\ref{eq2.22}). Note that
$${\frac{(z-\alpha )^{-j}-(w-\alpha )^{-j}}{z-w}}=-\begin{bmatrix} {\frac{1}{z-\alpha }}& \dots &{\frac{1}{(z-\alpha )^j}}\end{bmatrix}\begin{bmatrix} 0&\dots &1\\
\vdots & \ddots& \vdots\\ 1&\dots&0\end{bmatrix}\begin{bmatrix} {\frac{1}{w-\alpha}}\\ \vdots \\ {\frac{1}{(w-\alpha )^j}}\end{bmatrix}$$
in which the matrix $R$ is constant on cross diagonals, and is real symmetric and satisfies $R^2=I_{r\times r}$, with ${\hbox{trace}}\, R=0$ for even $r$ and $1$ for odd $r$, hence $R$ is similar via a real orthogonal matrix to a real diagonal matrix with signature $0$ or $1$, respectively.   
From the right-hand side of (\ref{eq2.31}) we extract coefficients depending on $w$ or $z$ such that
  \begin{align}\label{eq2.32}
    \delta_\Delta K(z,w)
       &= -\sum_{j=1}^N \psi_j(z)\phi_j(w)  \notag \\
       &= \delta_\Delta \int_0^\infty \sum_{j=1}^N \psi_j(z+u) \phi_j(w+u)\, du
  \end{align}
Also, $k(z,w)\rightarrow 0$ as $z\rightarrow\infty$ or $w\rightarrow\infty$. Likewise the right-hand side of (\ref{eq2.29}) converges to $0$ as  $z\rightarrow\infty$ or $w\rightarrow\infty$, so we have the identity (\ref{eq2.31}).
Recall that if $\int_0^\infty t \norm{\phi(t)}^2 \,dt$ is finite then $\Gamma_{\phi}$ is Hilbert--Schmidt. The representation $K = \sum_{j=1}^N \Gamma_{\psi_j}\Gamma_{\phi_j}$ then shows that $K$ is a trace-class operator.
\end{proof}

\begin{thm}\label{thmPDE} 
Let $k$ be a trace-class kernel as in Proposition~\ref{Prop2.8}, and let $k_x(u,v)=k(x+u,x+v)$. 
Then for all $\lambda\in \Cb$, there exists a matrix function $ T(x,y)\in M_{N\times N}(\Cb)$ for $0\leq x\leq y$ that satisfies
   \begin{equation}{\frac{\partial^2 T}{\partial x^2}}(x,y) -{\frac{\partial^2 T}{\partial y^2}}(x,y)=-2\Bigl({\frac{d}{dx}}T(x,x)\Bigr) T(x,y)\qquad (0<x<y)\end{equation}
and
   \begin{equation}{\frac{d}{dx}}\log\det (I+\lambda K_x)=\trace T(x,x)\qquad (x>0).\end{equation}
\end{thm}

\begin{proof} Using Proposition~\ref{Prop2.8}, write $K$ in the form $\sum_{j=1}^N \Gamma_{\psi_j}\Gamma_{\phi_j}$ with $\psi_j, \phi_j \in L^2((0, \infty ); \Rb)$, or in matrix form
\begin{equation}\begin{bmatrix} K&0\\ 0&0\end{bmatrix}=\begin{bmatrix} \Gamma_{\psi_1}&\dots &\Gamma_{\psi_N}\\ 0&\dots &0\end{bmatrix}\begin{bmatrix} \Gamma_{\phi_1}&0\\ \vdots &\vdots \\\Gamma_{\phi_N}&0\end{bmatrix}.\end{equation}  
 Consider $N\times N$ matrix functions
  \[ \Phi_1=
      \begin{bmatrix} \psi_1&\dots &\psi_N\\ {}&0&{}\end{bmatrix},
      \Phi_2=
      \begin{bmatrix}\phi_1&{}\\ \vdots &0 \\ \phi_N&{} \end{bmatrix}.
     \]
First we factorize $K$ as a product of matrix Hankel operators.
By Proposition \ref{Prop2.8} $K$ is a complex-linear combination of Hankel products $\Gamma_\psi\Gamma_\phi$, where $\psi, \phi\in L^2((0, \infty ); \Rb)$. 
Given measurable $\phi :(0, \infty )\rightarrow \Rb$ such that $\int_0^\infty t\vert \phi (t)\vert^2dt$ is finite, the Hankel operator with scattering function $\phi$ is self-adjoint and Hilbert--Schmidt, hence can be realised by a linear system $(-A,B,C)$ in continuous time with input and output space $\Cb$; this follows from \cite{MPT}.  
We can apply this to the real and imaginary parts of the entries of $\Phi_1$ and $\Phi_2$, and introduce for $j=1,2$ the matrices
  \[ (-A_j,B_j, C_j)
    = \Biggl( \begin{bmatrix} 
               -A_{j1} &\dots  &0\\
               \vdots  &\ddots &\vdots \\ 
                  0    &\dots  &-A_{jN}
             \end{bmatrix} , 
             \begin{bmatrix} 
              B_{j1}  & \dots & 0\\
              \vdots  &\ddots &\vdots \\
                 0    &\dots  &B_{jN}
             \end{bmatrix}, 
             \begin{bmatrix}
              C_{j1}  &\dots  & C_{jN}\\
              \vdots  &\dots  &\vdots \\
                 0    &\dots  & 0
              \end{bmatrix}\Biggr).  \] 
We can suppose that there are linear systems $(-A_1, B_1,C_1)$ and $(-A_2, B_2,C_2)$ with state space $H$ and input and output spaces $\Cb^N$ such that $\Phi_1(t)=C_1e^{-tA_1}B_1$ and $\Phi_2(t)=C_2e^{-tA_2}B_2$. Then 
   \begin{equation}\det (I+\lambda K)
     =\det (I+\lambda \Gamma_{\Phi_1}\Gamma_{\Phi_2})
     =\det \begin{bmatrix}I&\lambda\Gamma_{\Phi_1}\\
                       -\Gamma_{\Phi_2}&I\end{bmatrix}.
                       \end{equation}
Let $\bigl(-{\hat A}, {\hat B}, {\hat C}\bigr)$ be the linear system
   \[ \bigl(-{\hat A}, {\hat B}, {\hat C}\bigr)
      =\Biggl(\begin{bmatrix}-A_1&0\\ 0&-A_2\end{bmatrix},
       \begin{bmatrix}0&B_1\\ B_2&0\end{bmatrix},
       \begin{bmatrix}\lambda C_1&0\\ 0&-C_2\end{bmatrix}\Biggr)
       \]
which has scattering function
    \begin{equation}
    \hat \Phi (t)={\hat C}e^{-t{\hat A}}{\hat B}
    = \begin{bmatrix}0&\lambda C_1e^{-tA_1}B_1\\                
                    -C_2e^{-tA_2}B_2&0\end{bmatrix}
    =\begin{bmatrix} 0&\lambda \Phi_1(t)\\
                 -\Phi_2 (t)&0\end{bmatrix}.
    \end{equation}
The operator function
   \begin{equation}
   {\hat R}_x
   =\int_x^\infty e^{-t{\hat A}}{\hat B}{\hat C} e^{-t{\hat A}}\,dt
   = \int_x^\infty 
      \begin{bmatrix} 0&-e^{-tA_1}B_1C_2e^{-tA_2} \\
             \lambda e^{-tA_2}B_2C_1e^{-tA_1}&0\end{bmatrix}\, dt
    \end{equation}
satisfies the Lyapunov equation $\frac{d{\hat R}_x}{dx} = -{\hat R}_x {\hat A} - {\hat A}{\hat R}_x$ and we clearly have
   \begin{equation}\label{Lyap}
      -{\frac{d{\hat R}_x}{dx}}
      = e^{-x{\hat A}}{\hat B}{\hat C}e^{-x{\hat A}}
          \qquad (x>0).
    \end{equation}
Its determinant satisfies
   \begin{equation}
   \det (I+R_0)
      =\det (I+\Gamma_{\hat\Phi})
      =\det (I+\lambda K).
   \end{equation}
   
Let 
   \[ {\hat T}(x,y)
    =-{\hat C} e^{-x{\hat A}}(I+{\hat R}_x)^{-1} e^{-y{\hat A}}{\hat B} \]

To find this determinant, we consider the Gelfand--Levitan equation
   \begin{equation}\label{GL}
   {\hat \Phi} (x+y) + {\hat T}(x,y) 
      + \int_x^\infty {\hat T}(x,z) {\hat \Phi} (z+y) \, dz
       =0    \qquad (0<x<y)
    \end{equation}
which has solution
  \[ {\hat T}(x,y)
    =-{\hat C} e^{-x{\hat A}}(I+{\hat R}_x)^{-1} e^{-y{\hat A}}{\hat B} \] 
as one verifies by substituting the given formula into the equation. Also, the Lyapunov equation gives
   \begin{align*} 
   \trace\bigl( {\hat T}(x,x)\bigr)
      & = \trace \Bigl( (I+{\hat R}_x)^{-1}
          {\frac{d{\hat R}_x}{dx}}\Bigr)\\
      &= {\frac{d}{dx}} \log\det (I+{\hat R}_x).
    \end{align*}
Also, $\hat T$ is a solution of
  \[ {\frac{\partial^2 {\hat T}}{\partial x^2}}(x,y)
      -{\frac{\partial^2 {\hat T}}{\partial y^2}}(x,y)
      = -2\Bigl({\frac{d}{dx}}{\hat T}(x,x)\Bigr)
           {\hat T}(x,y)    \qquad (0<x<y). \]
\end{proof}
Theorem \ref{thmPDE} applies in particular to the sinh-Gordon equation, as we discuss in section \ref{S:Sinh-Gordon}.

\begin{ex} For the weight $x^\alpha e^{-x}$ on $(0, \infty )$, Laguerre's polynomials
$$L^{(\alpha )}_n(x)=x^{-\alpha} e^{x}{\frac{d^n}{dx^n}} \bigl( x^{n+\alpha }e^{-x}\bigr)\qquad (\alpha \geq 0, n=0,1, 2, \dots )$$
are orthononal. Let $u_\alpha (x)=e^{-x/2} x^{(\alpha +1)/2}L_n^{(\alpha )}(x)$
 and $w_\alpha (x)=u_\alpha '(x);$ then these give a solution of Laguerre's differential equation
$$\begin{bmatrix} 0&-1\\ 1&0\end{bmatrix} {\frac{d}{dx}}\begin{bmatrix} u_\alpha \\ w_\alpha \end{bmatrix} =\begin{bmatrix} {\frac{2n+\alpha+1}{2x}}+{\frac{1-\alpha^2}{4x^2}}-{\frac{1}{4}}&0\\ 0&1\end{bmatrix}\begin{bmatrix} u_\alpha \\ w_\alpha \end{bmatrix}, $$
so one can apply Proposition \ref{Prop2.8}. In particular, for $\alpha=1$, we obtain a Hankel product 
$${\frac{u_1(x)u_1'(y)-u_1'(x)u_1(y)}{x-y}}=\int_0^\infty {\frac{u_1(x+t)u_1(y+t)}{(x+t)(y+t)}}dt.$$
The Laguerre system is associated with orthogonal polynomials for the weight $x^\alpha e^{-x}$ on $(0, \infty )$, which is the limit of the semi-classical weights $x^\alpha (1-x/n)^n$ from \cite{Ma}. In section \ref{S:Hankeldeterminants}, we consider the singular weight $w_{\alpha ,s}(s)=x^\alpha e^{-x-s/x}$, which has $\log w_{\alpha ,s}(x)$ not integrable over $(0, 1)$, hence lies  beyond the scope of Szeg\"o's theory on asymptotic formulas for orthogonal polynomials.\par
\end{ex}

%
%

\section{A linear system for Darboux addition}\label{S:Darboux}

The pair of linear systems $(-A, B, C)$ giving $I+R_x$ and $(-A,B,-C)$ giving $I-R_x$ are related by a Darboux transformation, as discussed in Theorem 3.4 of \cite{BN}. As there, we introduce the matricial system
  \begin{equation} \bigl(-{\hat A}, {\hat B}, {\hat C}\bigr)
    =\Biggl(\begin{bmatrix}-A&0\\ 0&-A\end{bmatrix},
       \begin{bmatrix}0&B\\ B&0\end{bmatrix},
       \begin{bmatrix}C&0\\ 0&C\end{bmatrix}\Biggr) \end{equation}
which has scattering function
  \[ {\hat \Phi} (t)
   = {\hat C} e^{-t{\hat A}}{\hat B}
   = \begin{bmatrix} 0&Ce^{-tA}\\ Ce^{-tA}B&0\end{bmatrix}
   =\begin{bmatrix} 0&\phi (t)\\ \phi (t)&0\end{bmatrix}.
   \]
Let
  \begin{align}{\hat R}_x
    &=\int_x^\infty e^{-t{\hat A}}{\hat B}{\hat C} e^{-t{\hat A}} \,dt\nonumber\\
    &= \int_x^\infty \begin{bmatrix} 0&e^{-tA}BCe^{-tA} \\            
                        e^{-tA}BCe^{-tA}&0\end{bmatrix} \, dt= \begin{bmatrix} 0&R_x\\ R_x&0\end{bmatrix}.
   \end{align}
   so that when $I\pm R_x$ are invertible, we have
\[\hat F_x=(I+\hat R_x)^{-1} =\begin{bmatrix} (I-R_x^2)^{-1}& -R(I-R_x^2)^{-1}\\ -R(I-R_x^2)^{-1}& (I-R_x^2)^{-1}\end{bmatrix}.\]

We can then define
   \begin{align}\label{behatted}
     {\hat T}(x,y)
    &=-{\hat C} e^{-t{\hat A}}(I+{\hat R}_x)^{-1} e^{-y{\hat A}}{\hat B} 
    = \begin{bmatrix}W(x,y)&V(x,y)\\ V(x,y)& W(x,y)\end{bmatrix}
   \end{align}
where 
   \begin{align}
    W(x,y) & = Ce^{-xA}(I-R^2_x)^{-1}R_xe^{-yA}B \\
    V(x,y) & = -Ce^{-xA}(I-R_x^2)^{-1}e^{-yA}B.
   \end{align}
   
\begin{prop} 
\begin{enumerate}
\item[(i)] The functions $V(x,x)$ and 
  \[ S(x)=\log\det (I+R_x)-\log\det (I-R_x) \]
satisfy
  \begin{equation}
    {\frac{d}{dx}} S(x)=2V(x,x).
  \end{equation}
\item[(ii)] The Fredholm determinant satisfies
  \begin{equation} 
    \det (I-\Gamma_{\phi_x}^2)
      = \exp \Bigl( -4\int_x^\infty (t-x)V(t,t)^2\, dt\Bigr),
    \end{equation}
\item[(iii)] and with $q(x)=-4\lfloor A\rfloor$, 
the function
  \begin{equation} 
    U(x) = \exp \Bigl( 2\int_x^\infty V(t,t) \, dt\Bigr)
      =\exp (-S(x))
  \end{equation} 
satisfies
  \begin{equation} 
     {\frac{d^2U}{dx^2}} = q(x) U(x).
  \end{equation}
\end{enumerate}
\end{prop}

\begin{proof} (i) From Lyapunov's equation, we have
\begin{align}\label{S'} {\frac{d}{dx}}S(x)&={\frac{d}{dx}}\Bigl( \log\det (I+R_x)-\log\det (I-R_x)\Bigr)\nonumber\\
&= {\hbox{trace}}\Bigl( (I+R_x)^{-1}R'_x+(I-R_x)^{-1}R_x'\Bigr)\nonumber\\
&=-2{\hbox{trace}}\Bigl((I-R_x^2)^{-1}e^{-xA}BCe^{-xA}\Bigr)\nonumber\\
&=-2Ce^{-xA}(I-R_x^2)^{-1}e^{-xA}B\nonumber\\
&=2V(x,x).\end{align}

(ii) Then
\begin{equation}-V(x,x)=Ce^{-xA}(I-R_x^2)^{-1}e^{-xA}B=\bigl\lfloor (I-R_x)^{-1}(I+R_x)\bigr\rfloor\end{equation}
so by Lemma \ref{diffring}
\begin{align}\label{V'}-{\frac{d}{dx}}V(x,x)&={\frac{\partial}{\partial x}}\bigl\lfloor (I-R_x)^{-1}(I+R_x)\bigr\rfloor\nonumber\\
&= \bigl\lfloor -2A-2(I-R_x)^{-1} (AR_x+R_xA)(I-R_x)^{-1}\bigr\rfloor\nonumber\\
&=-2\bigl\lfloor A\bigr\rfloor -2Ce^{-xA}(I-R_x^2)^{-1}e^{-xA}BCe^{-xA}e^{-xA}(I-R_x^2)^{-1}e^{-xA}B\nonumber\\
&=-2\bigl\lfloor A\bigr\rfloor-2V(x,x)^2.\end{align}
and likewise
\begin{equation}2W(x,x)=Ce^{-xA}2R_x(I-R_x^2)^{-1}e^{-xA}B=\bigl\lfloor 2R_x(I-R_x)^{-1}(I+R_x)\bigr\rfloor\end{equation}
so 
\begin{align}2{\frac{d}{dx}}W(x,x)&={\frac{\partial}{\partial x}}\bigl\lfloor 2R_x(I-R_x)^{-1}(I+R_x)\bigr\rfloor\nonumber\\
&= \bigl\lfloor -4(I-R_x)^{-1}(AR_x+R_xA)(I-R_x)^{-1}\bigr\rfloor\nonumber\\
&= -4 Ce^{-xA}(I-R_x^2)^{-1}e^{-xA}BCe^{-xA}(I-R_x^2)^{-1}e^{-xA}B\nonumber\\
&=-4V(x,x)^2\end{align}

We have
\begin{align}-4V(x,x)^2&= 2{\frac{d}{dx}}W(x,x)\nonumber\\
&={\hbox{trace}}\, \hat T(x,z)\nonumber\\
&=-{\hbox{trace}}\,\hat Ce^{-x\hat A}(I+\hat R_x)^{-1}e^{-x\hat A}\hat B\nonumber\\
&=-{\frac{d}{dx}} {\hbox{trace}} (I+\hat R_x)^{-1} e^{-x\hat A} \hat B\hat Ce^{-x\hat A}\nonumber\\
&={\frac{d^2}{dx^2}} {\hbox{trace}} \log (I+\hat R_x)\nonumber\\
&={\frac{d^2}{dx^2}} \log\det (I+\hat R_x)\nonumber\\
&={\frac{d^2}{dx^2}}\log\det (I-R_x^2).\end{align}

(iii) Also 
  \[ U(x)=\exp 2\int_x^\infty V(t,t)dt \]
satisfies $U'(x)/U(x)=-2V(x,x)$, so as in Riccati's equation
\begin{align}
   {\frac{U''(x)}{U(x)}} -\Bigl( {\frac{U'(x)}{U(x)}}\Bigr)^2
   &= -2{\frac{d}{dx}}V(x,x)\nonumber\\
   &= -4\lfloor A\rfloor -4V(x,x)^2\nonumber\\
   &= -4\lfloor A\rfloor - \Bigl( {\frac{U'(x)}{U(x)}}\Bigr)^2
\end{align}
so $U''(x)=-4\lfloor A\rfloor U$.
\end{proof}

\begin{ex} With $\phi (x)={\hbox{Ai}}(x/2)$, we have
  \[ {\frac{d}{dx}}{\trace}\Gamma_{\phi_x}
  = {\frac{d}{dx}}\int_0^\infty {\hbox{Ai}}(x+t)dt=-{\hbox{Ai}}(x), \]
so as $x\rightarrow\infty$ we have 
\begin{align}
   2V(x,x)
   &= {\frac{d}{dx}}{\trace}\Bigl( \log (I+\Gamma_{\phi_x})-\log (I-\Gamma_{\phi_x})\Bigr)\nonumber\\
   &\approx {\frac{d}{dx}}2 {\trace} \Gamma_{\phi_x}\nonumber\\
   &= -2{\hbox{Ai}}(x).
\end{align}
This example is considered in detail by Hastings and McLeod \cite{HM}. 
\end{ex}

%
%

\section{Solutions of the sinh-Gordon equation}\label{S:Sinh-Gordon}

Howland \cite{Ho} observed that Hankel matrices are analogous to one-dimensional Schr\"od\-inger operators with the role of the Laplacian played by Carleman's operator 
 \begin{equation}\label{carleman}
   \Gamma f(x) = \int_0^\infty {\frac{f(y)dy}{x+y}}
         \qquad (f\in L^2(0, \infty )).
 \end{equation}
Power \cite{Power} gives several unitarily equivalent forms of this operator. To obtain such operators in terms of linear systems, we recall Proposition 2.1 from \cite{B2}. 
Let $H=L^2((0, \infty ); \Cb)$ which has dense linear subspace 
${\calD}(A) = \{ f\in L^2((0, \infty ); {\Cb}): yf(y)\in L^2 ((0, \infty ); \Cb)$. Suppose that $h\in L^\infty ((0, \infty ); {\mathbb R})$ and $h(y)/\sqrt{y}\in L^2(0, \infty )$. Then we let
   \begin{align}\label{SG-linear-sys}
  & A:{\mathcal D}(A)\rightarrow H:\quad  f(y)\mapsto yf(y);\nonumber\\
  & B_t: \Cb\rightarrow H:\quad  \beta\mapsto h(y)
      \exp \Bigl(-{\frac{t}{y}}\Bigr)\beta\nonumber\\
  & C_t : H\rightarrow \Cb:\quad  f\mapsto \int_0^\infty 
    \exp \Bigl(-{\frac{t}{y}}\Bigr) h(y)f(y) \, dy
  \end{align}
where the input and output operators depend upon the real parameter $t>0$. Then we introduce the scattering function
   \begin{equation}\label{phidef}
   \phi (x; t)
   =C_te^{-xA}B_t
   =\int_0^\infty \exp\Bigl( -xy-{\frac{2t}{y}}\Bigr) h(y)^2 \,dy
   \qquad (t>0)
   \end{equation}
which satisfies the linear PDE
\begin{equation}
  {\frac{\partial^2 \phi}{\partial x\partial t}}
    = 2\phi, 
\end{equation}
which may be interpreted as a linear counterpart of the sinh-Gordon equation. We also introduce
  \[ R_{(x;t)}=\int_x^\infty e^{-uA}B_tC_te^{-uA} \,du. \]
In the following result, the pair $(x,t)$ may be regarded as light-cone co-ordinates, rather than space and time. 
\begin{thm}\label{sinhGordontheorem} For the linear system (\ref{SG-linear-sys}), let
  \begin{equation}\label{Sphase} 
      S (x;t) =\log \det (I+R_{(x;t)}) -\log\det (I-R_{(x;t)}).
  \end{equation}
Then $S$ gives a solution of the sinh-Gordon equation
  \begin{equation}\label{sinhGordon} 
    {\frac{\partial^2S}{\partial x \, \partial t}}
     =2\sinh 2S . 
  \end{equation}
\end{thm}

\begin{proof} As an integral operator on $L^2(0, \infty )$, $R_{(x;t)}$ has a kernel 
  \begin{equation}\label{Rhowland}
    {\frac{h(y)h(z)}{y+z}}
      \exp \Bigl( -x(y+z)-t\Bigl( {\frac{1}{y}}+{\frac{1}{z}}\Bigr)\Bigr) \qquad (y,z>0)
    \end{equation}
which has the form of a Howland operator. This $R_{(x;t)}$ evidently defines a Hilbert--Schmidt linear operator on $L^2(0, \infty )$.
Suppose for simplicity that $h$ is real-valued. Then $R_{(x;t)}$ is the Schur product of the Carleman operator $\Gamma$ from (\ref{carleman}) with kernel $1/(y+z)$ and the tensor product of $h(y)e^{-xy-t/y}$ with itself. Power \cite{Power} showed that $\Gamma$ has spectrum $[0, \pi ]$, hence $R_{(x;t)}$ is self-adjoint and positive with trace
  \[ \trace\, R_{(x;t)}
  = \int_0^\infty {\frac{h(y)^2}{2y}} e^{_-2xy-2t/y} \,dy
   \leq {\frac{1}{2}} \int_0^\infty \frac{h(y)^2}{y}\, dy.\]
So $R_{(x;t)}$ is trace class and there exists $x_0>0$ such that for all $x>x_0$, we have $\Vert R_{(x;t)}\Vert <1$; hence $I\pm R_{(x;t)}$ are invertible. (Similar results hold for complex $h$ by polarization.)\par 
\indent Note that due to the special form of the linear system
\begin{align}\label{Rt} 
   -{\frac{\partial R_{(x;t)}}{\partial t}}
   &= A^{-1}R_{(x;t)}+R_{(x;t)}A^{-1}\nonumber\\
   &= A^{-1}(AR_{(x;t)}+R_{(x;t)}A)A^{-1}\nonumber\\
   &= -A^{-1}{\frac{\partial R_{(x;t)}}{\partial x}}A^{-1}\nonumber\\
   &= A^{-1}e^{-xA}B_tC_te^{-xA}A^{-1}
\end{align}
gives an operator of rank one so
\begin{align*} 
   {\frac{\partial S}{\partial t}}
    &= {\trace}\Bigl( (I+R_{(x;t)})^{-1}
           {\frac{\partial R_{(x;t)}}{\partial t}}
           +(I-R_{(x;t)})^{-1} {\frac{\partial R_{(x;t)}}{\partial t}}\Bigr)\\
    &= C_te^{-xA}A^{-1} (I-R_{(x;t)}^2)^{-1}A^{-1}e^{-xA}B_t.
\end{align*}
We write $2S(x;t)=\int_x^\infty \psi (s;t)ds$, 
where $\psi (x; t)=-4V(x,x)$ as in (\ref{S'}). By (\ref{V'}), we have
\begin{equation}\label{psix}
   2{\frac{\partial \psi}{\partial x}}
      =-16\lfloor A\rfloor -\psi^2
\end{equation}
so by integrating with respect to $x$ we deduce
  \[ -2\psi (x;t)=8{\frac{\partial }{\partial x}}\log\det (I+R_{(x;t)})-\int_x^\infty \psi (y;t)^2 \,dy, \]
which implies when we differentiate with respect to $t$ and multiply by $\psi$, that
\begin{equation}\label{stage1}
  -2\psi {\frac{\partial\psi}{\partial t}}
  = 8\psi {\frac{\partial^2}{\partial x\partial t}}
       \log\det (I+R_{(x;t)})
       -\psi (x;t) \int_x^\infty {\frac{\partial }{\partial t}} \psi(y;t)^2 \, dy.
\end{equation}
We also have
\begin{align}\label{psit} 
  {\frac{\partial \psi}{\partial t}}
   &= 2\lfloor A^{-1}(I+R_{(x;t)})\rfloor  \lfloor      (I+R_{(x;t)})A^{-1}\rfloor\nonumber\\
     &\qquad +2\lfloor (I-R_{(x;t)})^{-1}A^{-1}(I+R_{(x;t)})\rfloor\lfloor (I+R_{(x;t)})A^{-1}(I-R_{(x;t)})^{-1}\rfloor\nonumber\\
   &\qquad -4\lfloor (I-R_{(x;t)})^{-1}A^{-1}(I+R_{(x;t)})\rfloor-4\lfloor (I+R_{(x;t)}) A^{-1}(I-R_{(x;t)})^{-1}\rfloor,
\end{align}
and from (\ref{Rt}), 
\begin{align}\label{At}
   {\frac{\partial \lfloor A\rfloor}{\partial t}}
    &=  \lfloor A^{-1}(I+R_{(x;t)})\rfloor
        \lfloor (I+R_{(x;t)}) A^{-1}(I+R_{(x;t)})^{-1}A\rfloor\nonumber\\
  &\qquad +\lfloor A(I+R_{(x;t)})^{-1}A^{-1}(I+R_{(x;t)})\rfloor\lfloor   (I+R_{(x;t)})A^{-1}\rfloor\nonumber\\
  &\qquad-\lfloor (I+R_{(x;t)})A^{-1}(I+R_{(x;t)})^{-1}A\rfloor-\lfloor A(I+R_{(x;t)})^{-1}A^{-1}(I+R_{(x;t)})\rfloor .
\end{align}
Also, from (\ref{psix}) we have 
\begin{equation}\label{stage2}
   {\frac{\partial^2\psi}{\partial t\partial x}}
   =-8 {\frac{\partial \lfloor A\rfloor}{\partial t}}-\psi {\frac{\partial\psi}{\partial t}}.
\end{equation}
We need to combine  the equations (\ref{stage1}) and (\ref{stage2}) and thereby obtain the sinh-Gordon equation. 
The expression (\ref{psit}) involves factors $I-R$; whereas (\ref{At}) involves only $I+R$; we reconcile these by multiplying by $\psi =4\lfloor (I+R)/(I-R)\rfloor$, which cancels out the factors of $I-R$. 
From (\ref{Rt}), we have
\begin{align}
  {\frac{\partial^2}{\partial x\partial t}}\log\det (I+R_{(x;t)})
  &= -\lfloor A^{-1}(I+R_{(x;t)})\rfloor\lfloor (I+R_{(x;t)})A^{-1}\rfloor +\lfloor A^{-1}(I+R_{(x;t)})\rfloor\nonumber\\
  &+\lfloor (I+R_{(x;t)})A^{-1}\rfloor;
\end{align}
hence from the $\ast$ product (\ref{ast}) we obtain
\begin{align} 
    \psi &{\frac{\partial^2}{\partial x\partial x}}\log \det (I+R_{(x;t)})\nonumber\\
  &=-4\bigl\lfloor (I+R_{(x;t)})A^{-1}\bigr\rfloor\bigl\lfloor (I+R_{(x;t)})(I-R_{(x;t)})^{-1}\bigr\rfloor \bigl\lfloor A^{-1}(I+R_{(x;t)})\rfloor \nonumber\\
  &\qquad + 4\bigl\lfloor (I+R_{(x;t)})A^{-1}\bigr\rfloor\bigl\lfloor (I+R_{(x;t)})(I-R_{(x;t)})^{-1}\bigr\rfloor \nonumber\\
  &\qquad\qquad 
     +4\bigl\lfloor (I+R_{(x;t)})(I-R_{(x;t)})^{-1}\bigr\rfloor
      \bigl\lfloor A^{-1}(I+R_{(x;t)})A^{-1}\bigr\rfloor\nonumber\\
  &= -4\bigl\lfloor (I+R_{(x;t)})(I-R_{(x;t)})^{-1} -(I+R_{(x;t)})A^{-1}(I+R_{(x;t)})^{-1}A\bigr\rfloor\bigl\lfloor A^{-1}(I+R_{(x;t)})\bigr\rfloor\nonumber\\
  &\qquad +\psi - 4\bigl\lfloor (I+R_{(x;t)})A^{-1}(I+R_{(x;t)})^{-1}A\bigr\rfloor +\psi -4\bigl\lfloor A(I+R_{(x;t)})^{-1}A^{-1}(I+R_{(x;t)})\bigr\rfloor
\end{align}
hence we obtain
\begin{align} 
  \psi 
   &{\frac{\partial^2}{\partial x\partial x}}\log \det (I+R_{(x;t)})
       \nonumber\\ 
   &= 4\bigl\lfloor (I+R_{(x;t)})A^{-1}(I+R_{(x;t)})^{-1}A\bigr\rfloor\bigl\lfloor A^{-1}(I+R_{(x;t)})\bigr\rfloor\nonumber\\
   &\quad   -4\bigl\lfloor (I+R_{(x;t)})(I-R_{(x;t)})^{-1}\bigr\rfloor\bigl\lfloor A^{-1}(I+R_{(x;t)})\bigr\rfloor\nonumber\\
  &\qquad +\psi - 4\bigl\lfloor (I+R_{(x;t)})A^{-1}(I+R_{(x;t)})^{-1}A\bigr\rfloor +\psi -4\bigl\lfloor A(I+R_{(x;t)})^{-1}A^{-1}(I+R_{(x;t)})\bigr\rfloor\nonumber\\
  &= 4\bigl\lfloor (I+R_{(x;t)})A^{-1}(I+R_{(x;t)})^{-1}A\bigr\rfloor\bigl\lfloor A^{-1}(I+R_{(x;t)})\bigr\rfloor\nonumber\\
  &\qquad +4\bigl\lfloor A(I+R_{(x;t)})^{-1}A^{-1}(I+R_{(x;t)})\bigr\rfloor-\psi\nonumber\\
  &\qquad +\psi - 4\bigl\lfloor (I+R_{(x;t)})A^{-1}(I+R_{(x;t)})^{-1}A\bigr\rfloor +\psi -4\bigl\lfloor A(I+R_{(x;t)})^{-1}A^{-1}(I+R_{(x;t)})\bigr\rfloor\nonumber\\
  &= \psi +4\bigl\lfloor (I+R_{(x;t)})A^{-1}(I+R_{(x;t)})^{-1}A\bigr\rfloor\bigl\lfloor A^{-1}(I+R_{(x;t)})\bigr\rfloor\nonumber\\
  &\qquad  -4\bigl\lfloor (I+R_{(x;t)})A^{-1}(I+R_{(x;t)})^{-1}A\bigr\rfloor\nonumber\\
  &= \psi +2{\frac{\partial}{\partial t}\bigl\lfloor A\bigr\rfloor} .
\end{align}
Hence (\ref{stage1}) gives
\begin{equation}
  -2\psi {\frac{\partial\psi}{\partial t}}
   = 16{\frac{\partial }{\partial t}}\bigl\lfloor A\bigr\rfloor +8\psi -\psi (x;t)\int_x^\infty {\frac{\partial \psi^2}{\partial t}} \, dy
\end{equation}
so (\ref{stage2}) becomes
  \[ 2{\frac{\partial^2\psi}{\partial x\partial t}}
  = 8\psi -2\psi (x;t) \int_x^\infty \psi (y;t){\frac{\partial \psi (y;t)}{\partial t}} \, dy, \]
so
\begin{equation}\label{stage3}
  -{\frac{\partial \psi}{\partial t}}
      +\int_x^\infty \psi (y;t)\int_y^\infty
         \psi (v;t){\frac{\partial \psi (v;t)}{\partial t}} \, dvdy
   = 4\int_x^\infty \psi (y;t)\, dy,
\end{equation}
so we need to compare the integral in (\ref{stage3}) with the $\sinh 2S$ term in (\ref{sinhGordon}).

We introduce the integration operator $F_\psi: L^2(0, \infty )\rightarrow L^2(0,\infty )$, which depends upon $t>0$ by
  \[ F_\psi f(x;t)=\int_x^\infty \psi (s;t)f(s) \, ds.\]
Choosing $f=1$, we obtain by induction the formula
  \[ F_\psi ^{n}\circ 1 (x;t)
  = {\frac{1}{n!}}\Bigl(\int_x^\infty \psi (s;t)ds\Bigr)^{n}
  = {\frac{1}{n!}}(F_\psi \circ 1)^n(x;t).\]
Now we express the right-hand side of (\ref{sinhGordon}) as a composition of operators 
\begin{align*}
  \sinh 2S 
  &= \sum_{n=0}^\infty {\frac{2^{2n+1}S^{2n+1}}{(2n+1)!}}
    = \sum_{n=0}^\infty  {\frac{(F_\psi\circ 1)^{2n+1}}{(2n+1)!}}\\
  & =\sum_{n=0}^\infty F_\psi^{2n+1}\circ 1 (x;t)  
     = (I-F_\psi^2)^{-1}F_\psi\circ 1 (x;t).
\end{align*}
Hence the equation (\ref{stage3}) gives
  \[ -{\frac{\partial \psi}{\partial t}}=4(I-F_\psi^2)^{-1}F_\psi\circ 1 (x;t) \]
which gives
  \[ 2{\frac{\partial^2 S}{\partial t\partial x}}=4\sinh 2S.\]
which is the sinh-Gordon equation.
\end{proof}

\begin{cor}\label{detphase} Let $\phi_{(x;t)}(y)=\phi (y+2x;t)$, with $\phi$ as in Theorem \ref{sinhGordontheorem} for (i) and (ii). 
\begin{enumerate}
\item[(i)] Then
\begin{equation}\label{Sphasegamma} S(x;t)=\log \det (I+\Gamma_{\phi_{(x;t)}})-\log\det (I-\Gamma_{\phi_{(x;t)}})\end{equation}
satisfies the sinh-Gordon equation (\ref{sinhGordon}). 
\item[(ii)] The function $\varphi (z;t)=\log\det (I-\Gamma_{\phi_{(z;t)}}^2)$ is holomorphic on $\{ z: \Re z>x_0\}$ for some $x_0>0$ and $\varphi (z;t)\rightarrow 0$ as $\Re z\rightarrow\infty$. 
\item[(iii)] Now let
$$\phi (x;s/2)=\int_0^\infty e^{-xy-s/y}dy\qquad (\Re x>0, s>0). $$
Then the conclusion of (ii) holds for the associated Hankel operator $\Gamma_{\phi_{(z;t)}}$.
\end{enumerate}
\end{cor}

\begin{proof} (i) We have
$$\det (I+\lambda \Gamma_{\phi_{(x;t)}})=\det (I+\lambda \Theta_x^\dagger\Xi_x)=\det (I+\lambda \Xi_x\Theta_x^\dagger)=\det (I+\lambda R_{(x;t)})\qquad (\lambda\in {\mathbb C}),$$
so the formulas (\ref{Sphase}) and (\ref{Sphasegamma}) for $S(x;t)$ are equivalent, and the result holds by Theorem \ref{sinhGordontheorem}. See also \cite{TW2}.

\indent (ii) By Cauchy-Schwarz inequality
$$\vert \phi_{(z;t)}(x)\vert^2\leq \Vert h\Vert^2_\infty\int_0^\infty e^{-xy-4t/y}\vert h(y)\vert^2 dy\int_0^\infty e^{-(x+4\Re z)y}dy;$$
hence the Hilbert--Schmidt norm satisfies
\begin{align} \Vert \Gamma_{\phi_{(z;t)}}\Vert^2_{{\mathcal L}^2}&=\int_0^\infty x\vert \phi_{(z;t)}(x)\vert^2 dx\nonumber\\
&\leq \Vert h\Vert^2_\infty\int_0^\infty\Bigl\{  \int_0^\infty {\frac {xy}{x+ 4\Re z}}e^{-xy} dx \Bigr\}{\frac{e^{-4t/y}}{y}}\vert h(y)\vert^2 dy\end{align}
where the inner integral is bounded and converges to zero as $\Re z\rightarrow\infty$, so $\Vert \Gamma_{\phi_{(z;t)}}\Vert^2_{{\mathcal L}^2}\rightarrow 0$ by the dominated convergence theorem, hence the result. 

(iii) The conclusion of Corollary \ref{detphase}(ii) holds, but the hypotheses are not quite satisfied, so we provide a special argument. Let \begin{equation}\label{Kdef}
   K_1(z) = \int_0^\infty \exp (-z\cosh u)\cosh u \,du
        \qquad (\Re z>0)
\end{equation}
be the modified Bessel function of the third kind of order $1$, also known as MacDonald's function of the second kind, which is holomorphic on $\{ z\: \Re z>0\}$ with $\vert K_1(z)\vert\leq K_1(\Re z)$. This also satisfies
$$ {\frac{e^{-t}}{ t}}\leq K_1(t)\leq \sqrt{ {\frac{\pi}{2t}}} e^{-t} +{\frac{e^{-t}}{t}}\qquad (t>0)$$
by simple estimates following from $\cosh u=1+2\sinh^2(u/2)$. 
Then for $\alpha=0$, we have
  \begin{align}
    \phi (x;s/2)&=\sqrt{ {\frac{4s}{x}}} K_1(2\sqrt{sx})\end{align}
so $\int_0^\infty x\vert \phi (x+2x_0;s/2)\vert^2dx$ is finite for all $x_0,s>0$ and we can deduce $\Vert \Gamma_{\phi_{(x_0)}}\Vert_{{\mathcal L}^2}\rightarrow 0$ as $x_0\rightarrow\infty$, as in Corollary \ref{detphase}(ii).\par 
\end{proof}


%
%

\section{Hankel determinants}\label{S:Hankeldeterminants}

The Wishart ensemble is a standard model in random matrix theory \cite[page 91]{Fo} which produces the Laguerre ensemble of random eigenvalues on $(0, \infty )$. The Laguerre ensemble is in turn used as a model in the theory of wireless transmission, and in an integral model of quantum field theory at finite temperature \cite{CI}. The $KP$, $KdV$ and sinh-Gordon differential equation can be interpreted as aspects of a single hierarchy.

One can produce particular solutions to such differential equations in terms of the Painlev\'e transcendental functions. Chen and Its \cite{CI} considered the singularly perturbed Laguerre weight $y^\alpha e^{-y-s/y}$ for $y>0$ and $\alpha >0$. Let $h(y)=y^{\alpha/2}$, and observe that 
  \begin{equation}\label{scattering} \phi (x;s/2)=\int_0^\infty y^\alpha e^{-xy-s/y} \, dy\end{equation}
is a moment generating function for moments of this weight, so we have
 \[ (-1)^{j+k}\Bigl({\frac{\partial^{j+k}}{\partial x^{j+k}}}\Bigr)_{x=1}
    \phi (x;s/2)
    =\int_0^\infty y^{\alpha +j+k}e^{-y-s/y} \,dy. \]
Then the corresponding Hankel determinants are given by 
\begin{equation}\label{Hankeldet} D_n(s)
    =\det\Bigl[ \int_0^\infty y^{\alpha +j+k}e^{-y-s/y} \,dy \Bigr]_{j,k=0}^{n-1}
     \qquad (s>0; n=1,2, \dots), \end{equation} 
which turns out to be the isomonodromic tau function for a particular sequence of solutions of Painlev\'e $\text{III}'$ differential equation. This terminology refers to the work of Okamoto.  
\begin{ex} Let $s=0$ and $\alpha>0$, let $G(z)$ be Barnes $G$-function such that $G(0)=1$ and $G(z+1)=\Gamma (z)G(z)$, where $\Gamma (z)$ is Euler's gamma function. Then by (1.19) of \cite{CI}
$$D_n(0)={\frac{G(n+1)G(n+\alpha +1)}{G(\alpha +1)}}.$$
\end{ex}
Let $\alpha =0$. By Andr\'eief's identity \cite{F}, we have
  \begin{multline}
  \det\bigl[ (-1)^{j+k-2}K_1^{(j+k-2)}(t)\bigr]_{j,k=1}^n\\
     ={\frac{1}{n!}} \int_0^\infty\dots \int_0^\infty 
        \prod_{1\leq j<k\leq n} (\cosh u_j-\cosh u_k)^2 
        \prod_{j=1}^n e^{-t\cosh u_j}\cosh u_j \, du_j
  \end{multline}
with $K_1$ as in (\ref{Kdef}). With the change of variable $x_j=2/(1+\cosh u_j)$, this may be written as
   \[ {\frac{1}{n!}} \int_0^\infty\dots \int_0^\infty
     \exp\Bigl( -t\Bigl({\frac{2}{x_j}}-1\Bigr) \Bigr) 
     \prod_{j=1}^n {\frac{x_j^{-2n}(2-x_j)}{\sqrt{1-x_j}}}
     \prod_{1\leq j<k\leq n}(x_j-x_k)^2\prod_{j=1}^n dx_j .
     \]
This is associated with the generalized unitary ensemble for the scalar potential
    \begin{equation}\label{potential}
    u_n(x) = t\Bigl( {\frac{2}{x}}-1\Bigr) +2n\log x+(1/2)\log (1-x) -\log (2-x),
    \end{equation}
with
    \begin{equation} u_n'(x) = {\frac{-2t}{x^2}} +{\frac{2n}{x}} -{\frac{1/2}{1-x}}+{\frac{1}{2-x}},
    \end{equation}
and
    \begin{equation} u_n''(x) = {\frac{4t}{x^3}} +{\frac{-2n}{x^2}} -{\frac{1/2}{(1-x)^2}}+{\frac{1}{(2-x)^2}},
    \end{equation}
    
Dyson introduced a technique for funding the asymptotics of such determinants, which has been refined by Chen \cite[page 4603]{CM} and others into the Coulomb fluid method. In this, the point distribution $n^{-1}\sum_{j=1}^n \delta_{x_j}$ of the $x_j$ is approximated in the weak topology on probability measures as $n\rightarrow\infty$ by a continuous distribution $\rho$ which is found by potential theory. We write
  \[ \HT \rho (x) = \pv\int_{-\infty}^\infty {\frac{\rho (y)}{x-y}}{\frac{dy}{\pi}} \]
for the Hilbert transform of $\rho$. 

Fix $\xi \in (0,1)$ and write $t$ for $2n \xi$.
The zeros of $u_n'(x)$ may be approximated by the zeros of $h(x)=(2n-1/2)(x-1)(x-2)(x-t/n)$ by Rouch{\'e}'s theorem, so we locate a real zero of $u_n'(x)$ near to $2\xi$. Note that 
  \[ {\frac{u_n''(2\xi )}{n}}
  = {\frac{1}{2\xi^2}}-{\frac{1}{2n(1-2\xi )^2}}
         +{\frac{1}{n(2-\xi )^2}}\] 
is positive for all $0<\xi <1/2$ and all sufficiently large $n$, so $u_n$ is convex near to this minimizer. 
For a continuous $V:[0, 1]\rightarrow {\mathbb R}$, we introduce the energy functional
   \begin{equation} 
   E_{V}(\rho )
   =\Bigl\{ \int_0^1 V(x)\rho (x) \, dx
     + \int_0^1\int_0^1 \log {\frac{1}{\vert x-y\vert}} \rho (x)\rho (y) \, dxdy \Bigr\},
   \end{equation}
   where $V$ indicates an electric field and the $\log \vert x-y\vert$ term involves electrostatic interaction between points $x,y\in [0,1]$, where the charge distribution is $\rho$. The equilibrium distribution is defined to be the minimizer of this energy functional; see \cite{ST}.
   
We consider the approximate scaled potential that is given by the first two summands in (\ref{potential}), namely
   \begin{equation}\label{approxpotential}
     u_0(x) =2\xi \Bigl( {\frac{2}{x}}-1\Bigr) +2\log x,
   \end{equation}
with
  \[u_0'(x) ={\frac{-2\xi}{x^2}}+{\frac{2}{x}}; \quad u_0''(x)={\frac{8\xi }{x^3}}-{\frac{2}{x^2}}\qquad (0<x<1)\]
which has a local minimum at $x=2\xi$. 

\begin{prop} 
\begin{enumerate}
\item[(i)] The equilibrium distribution for $u_0$ is $\sigma_0$, where 
\[
  \sigma_0(x)
    = {\frac{\sqrt{(b-x)(x-a)}}{\pi \sqrt{ab}}}
        \Bigl( {\frac{\xi (a+b)}{abx}} +{\frac{2\xi}{x^2}}-{\frac{1}{x}}\Bigr)
        \qquad (a<x<b),
  \]
is supported on $(a,b)$, where 
   \[ a,b
     = {\frac{4\pi\xi}{(2\pi -1)^2}}\bigl( 2\pi \pm \sqrt{4\pi -1}\bigr).
   \]
\item[(ii)] For $1/4<\xi<1/2$, the free logarithmic Sobolev inequality 
   \begin{equation}
   E_{u_0}(p )-E_{u_0}(\sigma_0)
     \leq {\frac{2}{8\xi -2}}\int_a^b (2\pi \HT p(x)-u_0'(x))^2p(x) \,dx.
  \end{equation}
holds for all probability density functions $p\in L^3(a,b)$.
\item[(iii)] 
Let $\rho_n$ be the minimizer of $E_{u_n/n}(\rho )$ over all continuous $\rho$ such that $\rho \geq 0$ and $\int_0^1 \rho (x)dx=1$.
Suppose that $\rho_n$ has a continuous density which is supported in a single interval. Then $\rho_n$ converges weakly as $n\rightarrow\infty$ to $\sigma_0$.  
\end{enumerate}
\end{prop}

\begin{proof} (i) By standard results \cite{ST} Theorem 1.3 and p. 215, there exists a unique continuous probability density function on $[0,1]$ that attains the minimum of $E_V(\rho )$ for continuous $V$.  
We aim to solve the singular integral problem 
   \[ u_0'(x)=\pv \int_a^b {\frac{2\sigma_0 (y)}{x-y}} \,dy 
      = 2\pi \HT \sigma_0(x) \]
for the equilibrium density $\sigma_0$ subject to the constraint $\int_a^b \sigma_0 (y) \,dy=1$. Since $u_0$ is convex, the solution is continuous and positive on a single interval 
$(a,b)$, where $0<a<b<1$, and we have
   \begin{align}\label{equilibdensity} 
   \sigma_0 (x)
   &={\frac{\sqrt{(b-x)(x-a)}}{2\pi^2}}
      \int_a^b {\frac{u_0'(x)-u_0'(y)}{x-y}}{\frac{dy}{\sqrt{(b-y)(y-a)}}}\nonumber \\
   &={\frac{\sqrt{(b-x)(x-a)}}{2\pi^2}}
      \int_a^b \Bigl( {\frac{4\xi (x+y)}{x^2y^2}} - {\frac{2}{xy}}\Bigr)
             {\frac{dy}{\sqrt{(b-y)(y-a)}}}\nonumber  \\
   &= {\frac{\sqrt{(b-x)(x-a)}}{\pi \sqrt{ab}}}
        \Bigl( {\frac{\xi (a+b)}{abx}} +{\frac{2\xi}{x^2}}-{\frac{1}{x}}\Bigr)
        \qquad (a<x<b),
  \end{align}
where the final step follows from the substitution $y=a+(b-a)\sin^2\phi$ and some elementary integrals. The initial factor has the form of a semicircular distribution, which is modulated by a rational function with poles at $x=0$ outside of the support interval $(a,b)$. The endpoints of this interval are subject to the constraints
  \[
    0 = \int_a^b {\frac{u_0'(x)\, dx}{\sqrt{(b-x)(x-a)}}}
    = -{\frac{2\pi\xi (a+b)}{ab\sqrt{ab}}}+{\frac{2\pi}{\sqrt{ab}}},
  \]
and
  \[
 1 = \int_a^b {\frac{xu_0'(x)\, dx}{\sqrt{(b-x)(x-a)}}}
    = 2\pi -{\frac{4\pi\xi}{\sqrt{ab}}},
  \]
so that 
  \begin{equation}\label{ab}
   a,b={\frac{2\pi^2\xi}{(\pi -1/2)^2}}\mp {\frac{2\pi \xi}{\pi -1/2}}\sqrt{ {\frac{\pi^2}{(\pi -1/2)^2}}-1}.
   \end{equation}

(ii) We have $u_0''(x)=(8x\xi -2x)/x^3>8\xi -2>0$. The result follows from Theorem 3.1 of \cite{BP}. 
   
(iii) Returning to the original potential, we have $u_n=nu_0+f$ where the correction term $f(x)=(1/2)\log (1-x)-\log (2-x)$. There is a corresponding correction to the equilibrium density, $\rho_n =\sigma_0 +\tilde\rho/n$ where
   \begin{align*} 
   {\tilde\rho} (x)
   &= {\frac{1}{2\pi^2\sqrt{(b-x)(x-a)}}} 
      \pv \int_a^b {\frac{\sqrt{(b-y)(y-a)}}{y-x}} f'(y) \,dy \\
   &= {\frac{1}{2\pi^2\sqrt{(b-x)(x-a)}}} 
      \ \pv \int_a^b {\frac{\sqrt{b-y)(y-a)}}{y-x}}
           \Bigl( {\frac{-1/2}{1-y}}+{\frac{1}{2-y}}\Bigr) \,dy \\
   &= {\frac{1}{4\pi\sqrt{(b-x)(x-a)}}}
        \Bigl( {\frac{\sqrt{(1-a)(1-b)}}{1-x}} 
                    +1+{\frac{2\sqrt{(2-a)(2-b)}}{x-2}}\Bigr)
  \end{align*}
where all the roots are positive square roots and we have used the evaluation (258) in the Appendix to \cite{CM}. This has the form of an arcsine distribution, modulated by a rational function with poles outside $(a,b)$.
\end{proof}

Biane \cite{BP} interprets the tangent space to $L^2(\sigma_0)$ as a space of functions with finite Dirichlet norm, which has a natural interpretation in the present context. Let $\psi :(a,b)\rightarrow \Rb$ be a differentiable function such that 
\begin{equation} 
   \sigma (\psi )^2
   = {\frac{1}{2\pi^2}} \int_a^b {\frac{\psi (x)}{\sqrt{(b-x)(x-a)}}}
    \ \pv \int_a^b {\frac{\sqrt{(b-y)(y-x)}}{x-y}} \psi '(y) 
       \, dy dx,
\end{equation}
is finite. Then by results of \cite{CL}, the asymptotic distribution of a linear statistic $\sum_{j=1}^N \psi (x_j)$ is a Gaussian with mean $N\int_a^b \psi (x)\sigma_0(x)\,dx$ and variance
$\sigma (\psi )^2$ where $\sigma (\psi)$ depends upon $[a,b]$ but not upon the equilibrium distribution $\sigma_0$ itself. In terms of the Chebyshev polynomials, we introduce $\psi ((a+b)/2(+(b-a)t/2) = \sum_k a_k T_k(t)$ where
\[ {\frac{a_k}{2}}=\int_{-1}^1 \psi ((a+b)/2 +(b-a)t/2) T_k(t){\frac{dt}{\pi \sqrt{1-t^2}}}\]
so the variance is $\sigma (\psi )^2=\sum_{k=1}^\infty ka_k^2/4.$ Then we introduce $\phi (\theta )=\psi ((a+b)/2 +(b-a) \cos\theta /2)$ has an expansion 
  \[ \phi (\theta )=(1/2)a_0+\sum_{k=1}^\infty a_k \cos k\theta, \]
which gives a self-adjoint and Hilbert--Schmidt Hankel matrix $\Gamma_\phi =[a_{j+k}]_{j,k=0}^\infty$.

\begin{rem}
In their study of the Laguerre unitary ensemble, Forrester and Witte \cite{FW} obtain solutions of $\mathrm{PIII}'$ from similar formulas involving the modified Bessel function of the first kind $I_\nu $, 
although their Corollary 4.5 has a Toeplitz rather than a Hankel determinant.\par
\end{rem}
\begin{ex} For $s,x,\alpha >0$ in (\ref{scattering}), consider
$$v(z)=-\alpha \log z +\sqrt{sx}\Bigl(z+{\frac{1}{z}}\Bigr)\qquad (z>0).$$
By a simple scaling, this is the scalar potential that one needs to study the determinant (\ref{Hankeldet}) with weight
$y^\alpha e^{-y-s/y}$ on $(0, \infty )$. Then $v$ is convex with
\begin{equation}\label{diffvee}{\frac{v'(z)-v'(y)}{z-y}}={\frac{\alpha}{z}}+\sqrt{sx}\Bigl({\frac{z+y}{z^2y^2}}\Bigr)\end{equation}
a sum of products of rational functions of $z$ and $y$, and there exist $b>a>0$, such that the integral equation
\begin{equation}\label{inteq}v(z)=2\int_a^b \log\vert z-y\vert\rho (y) dy+C\qquad (z\in (a,b))\end{equation}
for some $C\in {\mathbb R}$ with the normalization $\rho (y)\geq 0$ and
$$\int_a^b \rho (y)\,dy=1,$$
has solution
$$\rho (z)={\frac{\sqrt{(b-z)(z-a)}}{2\pi\sqrt{ab}}}\Bigl[ {\frac{\alpha }{z}}+{\frac{\sqrt{sx}}{z^2}}+{\frac{\sqrt{sx}}{2z}}\Bigl( {\frac{1}{a}}+{\frac{1}{b}}\Bigr)\Bigr]\qquad (a<z<b)$$
by a similar argument to (\ref{equilibdensity}). We need $a>0$ to ensure that the resulting $\rho$ is integrable. The relevant integrals arise from (248) in the appendix to \cite{CM}. Let $(P_{j,N})_{j=0}^\infty$ be the monic polynomials that are orthogonal for the weight $e^{-Nv(z)}$, which gives rise to the integral equation (\ref{inteq}) multiplied through by $N$, with the normalizations preserved. Then using the results of \cite{CL}, we observe that $N\int_a^b \rho (x)\, dx$ gives the number of zeros of $P_{N,N}(z)$ and have an asymptotic formula
$$\log P_{N,N}(z)\sim N\int_a^b \log (z-y)\rho (y)dy\qquad (N\rightarrow\infty )$$
for $z\in {\mathbb C}\setminus {\mathbb R}.$ This type of double scaling is standard in random matrix theory and addresses the singularity of the weight; see section 4 of \cite{CL} for more details.
\end{ex}

\vskip0.1in
\noindent {\bf Acknowledgements.} Part of this document was written by GB during a visit to the University of Macau, supported by a visiting scholarship. The results were announced at IWOTA Chapman USA, and were partially supported by EPSRC Grant EP/T007524/1 IWOTA Lancaster 2020. The authors acknowledge helpful correspondence with S.J.A. Malham leading to the proof of Theorem \ref{sinhGordontheorem}.

\vskip0.1in
\noindent {\bf Rights Retention Statement.} This research was produced in whole or part by UNSW Sydney researchers and is subject to the UNSW Intellectual property policy. For the purposes of Open Access, the authors have applied a Creative Commons Attribution CC-BY licence to any Author Accepted Manuscript(AAM) version arising from this submission.
\vskip.01in
\noindent {\bf Declarations of interest:none}


\end{document}